\documentclass[11pt]{article}
\usepackage{amsmath}
\usepackage{amsthm}
\usepackage{latexsym}
\usepackage[noblocks]{authblk}
\usepackage{hyperref}
\usepackage{enumerate}
\usepackage{amssymb,amsmath}
\usepackage{multimedia}
\usepackage{subfigure}
\usepackage{authblk}
\usepackage{mathrsfs} 
\usepackage{epsfig}
\usepackage{mathdots}
\usepackage{caption}
\usepackage{color}
\usepackage{kotex}
\usepackage{pgfpages,tikz,tikz-cd,pgfkeys,pgfplots}
\usepackage{graphicx}

\usetikzlibrary{arrows,positioning,matrix,fit,backgrounds,shapes}

\usetikzlibrary{fit,matrix}
\tikzset{
mN/.style = {
    draw=#1, semithick, inner sep=0pt}
             }

\usetikzlibrary{arrows,positioning,matrix,fit,backgrounds,shapes,shapes.geometric, intersections}
\usetikzlibrary{shapes.callouts,decorations.pathmorphing, patterns}
\usetikzlibrary{decorations.markings} 
\tikzstyle{vertex}=[circle, draw, inner sep=0pt, minimum size=6pt] 

\newtheorem{theorem}{Theorem}[section]
\newtheorem{lemma}[theorem]{Lemma}
\newtheorem{proposition}[theorem]{Proposition}
\newtheorem{corollary}[theorem]{Corollary}

\theoremstyle{definition}

\newtheorem{remark}[theorem]{Remark}

\definecolor{LemonChiffon}{rgb}{100, 98, 80}
\definecolor{myblue}{rgb}{0,0.4,0.8}
\definecolor{orange}{rgb}{1, 0.4, 0}
\definecolor{mygreen}{rgb}{0, 0.8, 0}
\definecolor{myred}{rgb}{204, 0, 0}
\definecolor{violet}{RGB}{0.4,0.2,1}
\definecolor{brown}{rgb}{0.6, 0.4, 0}

 \newcounter{statement}
\newcommand{\statement}[2]{%
 \begin{equation}\refstepcounter{statement}\tag{S\thestatement}\label{#1}%
  \parbox{\dimexpr\linewidth-4em}{#2}%
 \end{equation}%
}
\begin{document}

\title{Competition-common enemy graphs of degree-bounded digraphs}
\author{Myungho Choi$^{1}$, Hojin Chu$^{1}$, Suh-Ryung Kim$^{1}$ \\
 {\footnotesize $^{1}$ \textit{Department of Mathematics Education,
Seoul National University,}}\\{\footnotesize\textit{
Seoul 08826, Rep. of Korea}}\\
{\footnotesize\textit{
nums8080@snu.ac.kr, ghwls8775@snu.ac.kr, srkim@snu.ac.kr}}\\
{\footnotesize}}
\date{}\maketitle

\begin{abstract}
The competition-common enemy graph (CCE graph) of a digraph $D$ is the graph with the vertex set $V(D)$ and an edge $uv$ if and only if $u$ and $v$ have a common predator and a common prey in $D$.
If each vertex of a digraph $D$ has indegree at most $i$ and outdegree at most $j$, then $D$ is called an $\langle i,j \rangle$ digraph.
In this paper, we fully characterize  the CCE graphs of $\langle 2,2\rangle$ digraphs. 
Then we investigate the CCE graphs of acyclic $\langle 2,2 \rangle$ digraphs, and prove that any CCE graph of an acyclic $\langle 2,2 \rangle$ digraph with at most seven components is interval, and the bound is sharp.
While characterizing acyclic $\langle 2,2 \rangle$ digraphs that have interval graphs as their competition graphs, Hefner~{\it et al}. (1991) initiated the study of competition graphs of degree-bounded digraphs. 
 Recently, Lee~{\em et al}. (2017) 
and Eoh and Kim (2021) studied phylogeny graphs of degree-bounded digraphs to extend their work.
\end{abstract}

{\small
 }

\section{Introduction}
In this paper, all the graphs and digraphs are assumed to be simple.
For all undefined graph-theoretical terminology, we follow~\cite{Bondy2008gt}.
Given a digraph $D$,  the \textit{competition graph} of $D$ has the same vertex set as $D$ and an edge $uv$ if and only if there is a vertex $x$ such that $(u,x)$ and $(v,x)$ are arcs of $D$.
The notion of competition graphs was introduced by Cohen~\cite{Cohen1968} in connection with competition in the food web of ecosystems.
In this vein, if $(u,x)$ is an arc of a digraph, then
 $x$ and $u$ are called a {\em prey} of $u$ and a {\em predator} of $x$, respectively.
The competition graph has applications in coding, radio transmission, and modeling of complex economic systems and has been studied in various variations.
The notion of competition-common enemy graphs (CCE graphs) was introduced by Scott~\cite{Scott1987tc} as one of such variants.
The \textit{CCE graph} of $D$, denoted by $CCE(D)$, has the same vertex set as $D$ and an edge $uv$ if and only if there are vertices $x$ and $y$ such that $(y,u)$, $(y,v)$, $(u,x)$, and $(v,x)$ are arcs of $D$.
Thus the CCE graph of $D$ has an edge $uv$ if and only if $u$ and $v$ have a common prey and a common predator in $D$.
For results on CCE graphs, readers may refer to \cite{sano2015double}, \cite{Kim2007oc}, \cite{Sano2010tc}, \cite{Bak1997oc}, \cite{Fisher2001tf}, \cite{Jones1987sr}, \cite{Bak1996ot}, and \cite{furedi1998double}.

In 1991, Hefner~{\it et al.}~\cite{Hefner1991ic} introduced the notion of $(i,j)$ digraphs and studied competition graphs of $(i,j)$ digraphs, which they call {\it $(i,j)$ competition graphs}.
The \textit{$(i,j)$ digraph} is an acyclic digraph $D$ such that $d^{-}(x) \le i$ and $d^{+} (x) \le j$ for every vertex $x$ in $V(D)$, where $d^{-}(x)$ and $d^{+}(x)$ denote the indegree and outdegree of a vertex $x$, respectively.
They claimed that limitation on the indegree or outdegree of a vertex is reasonable based on the empirical results of Cohen and Briand~\cite{Cohen1984tl} which suggest that in a food web, the number of arcs incident to each species is quite small in an average sense and it is actually about $2$.
Recently, Lee~{\em et al}.~\cite{lee2017phylogeny}
and Eoh and Kim \cite{eoh2021chordal} studied phylogeny graphs of degree-bounded digraphs to extend their work.

Under degree restrictions, Hefner~{\it et al.}~\cite{Hefner1991ic} characterized the $(2,2)$ competition graphs and also studied the $(2,2)$ digraphs whose competition graphs are interval graphs as follows.
Here, an {\it interval graph} is a type of graph where each vertex represents an interval on the real number line, and there is an edge between two vertices if and only if their corresponding intervals intersect.
\begin{theorem}[\cite{Hefner1991ic}]\label{thm:competition}
A graph is a $(2,2)$ competition graph if and only if each component is an isolated vertex, a path, or a cycle, and the number of isolated vertices is at least $2$ if every component is a cycle of length $> 3$ and at least $1$ otherwise.
\end{theorem}

\begin{theorem}[\cite{Hefner1991ic}]\label{thm:forbidden digraph}
	There is a forbidden subdigraph characterization of $(2,2)$ digraphs whose competition graphs are interval. 
\end{theorem}

We obtain a result (Theorem~\ref{thm:sumup}) on CCE graphs similar to Theorem~\ref{thm:competition} by loosening the condition ``acyclic".

The \textit{$\langle i,j \rangle$ digraph} is a digraph satisfying $d^{-}(x) \le i$ and $d^{+} (x) \le j$ for every vertex $x$.
Given a graph $G$, we say that $G$ is an {\it $\langle i,j\rangle$ CCE graph} if it is the CCE graph of an $\langle i,j\rangle$ digraph.
In a similar vein, we say that a graph is an {\it $(i,j)$ CCE graph} if it is the CCE graph of an $(i,j)$ digraph.

From now on, we identify an isolated vertex with a trivial path.

\begin{theorem}\label{thm:sumup}
	The following are equivalent.
\begin{enumerate}[(a)]
	\item A graph $G$ is a $\langle 2,2 \rangle$ CCE graph.
	\item Each component of $G$ is a path or a cycle. Especially, when the number of path components in $G$ is one, the path component must be trivial.
	\end{enumerate}
	\end{theorem}
	Given a graph,
an induced cycle of length four or more is called a {\it hole}.
Since a path or a triangle is interval and a hole is not interval,
we obtain the following.
\begin{corollary}\label{cor:interval}
	The following are equivalent.
\begin{enumerate}[(a)]
	\item A graph $G$ is an interval $\langle 2,2 \rangle$ CCE graph.
	\item Each component of $G$ is a path or a triangle. Especially, when the number of path components in $G$ is one, the path component must be trivial.
	\end{enumerate}
	\end{corollary}

By definition, an $(i,j)$ digraph is an $\langle i,j \rangle$ digraph.
Therefore $(a) \Rightarrow (b)$  of Theorem~\ref{thm:sumup} is valid for $(2,2)$ CCE graphs, that is, 
\begin{itemize}
	\item[($\star$)] if a graph is a $(2,2)$ CCE graph, then each component is a path or a cycle.
\end{itemize}
A {\it chordal graph} is a graph without holes. 
It is well known that an interval graph is a chordal graph.

Since a path or a triangle is an interval graph, it is true that a $(2,2)$ CCE graph is chordal if and only if it is interval.
Moreover, if $H$ is a hole in the CCE graph of a subdigraph of a given $(2,2)$ digraph, then ($\star$) guarantees that $H$ is still a hole in the CCE graph of a given digraph.
Therefore we may conclude that Theorem~\ref{thm:forbidden digraph} is true for $(2,2)$ CCE graphs.

We note that $(b) \not \Rightarrow (a)$ of Theorem~\ref{thm:sumup} for $(2,2)$ CCE graphs.
For example, 
a triangle satisfies part (b) of Theorem~\ref{thm:sumup} but it cannot be a $(2,2)$ CCE graph since it contains no isolated vertices.
Accordingly, we investigated which graphs satisfying part (b) of Theorem~\ref{thm:sumup} are intervals and could prove that all $(2,2)$ CCE graphs with up to seven components are interval.
By demonstrating the existence of a $(2,2)$ CCE graph with eight component that is not interval, 
we showed that our result is optimal in some aspects. 
\begin{theorem}\label{thm:interval}
Let $G$ be a $(2,2)$ CCE graph and $t$ be the number of components of $G$.
If $t\leq 7$, then $G$ is an interval graph.
Further, the inequality is tight.
\end{theorem}
Our main results are summarized in Figure~\ref{fig:Venn}. 
The following is a table indicating which region represents a specific set in the figure.
\vspace{-0.5cm}
\begin{center}
\resizebox{\columnwidth}{!}{
\begin{tabular}{l|l}
\multicolumn{1}{c|}{ The region}& \multicolumn{1}{c}{represents the set of}  \\  \hline
enclosed by solid line rounded rectangle (A) & disjoint union of paths and cycles \\ 
shaded with a checkered pattern (B) & $\langle 2,2 \rangle$ CCE graphs \\ 
vertically shaded (C) &  disjoint union of exactly one nontrivial path and cycles \\ 
enclosed by dotted line rounded rectangle (D) &  $(2,2)$ CCE graphs \\ 
enclosed by dashed line rounded rectangle  (E) &graphs with at most seven components  \\ 
\end{tabular}}
\end{center}
By Theorem~\ref{thm:sumup}, the region A is divided into the regions B and C. 
By Theorem~\ref{thm:interval}, the overlapping area of D and E is contained in the region representing the set of interval graphs. 
In Section~\ref{sec:Preliminaries}, we develop useful tools to be utilized in proving our main results.
We prove Theorems~\ref{thm:sumup} and~\ref{thm:interval} in Sections~\ref{sec:Only path components} and~\ref{sect:(2,2)}, respectively.
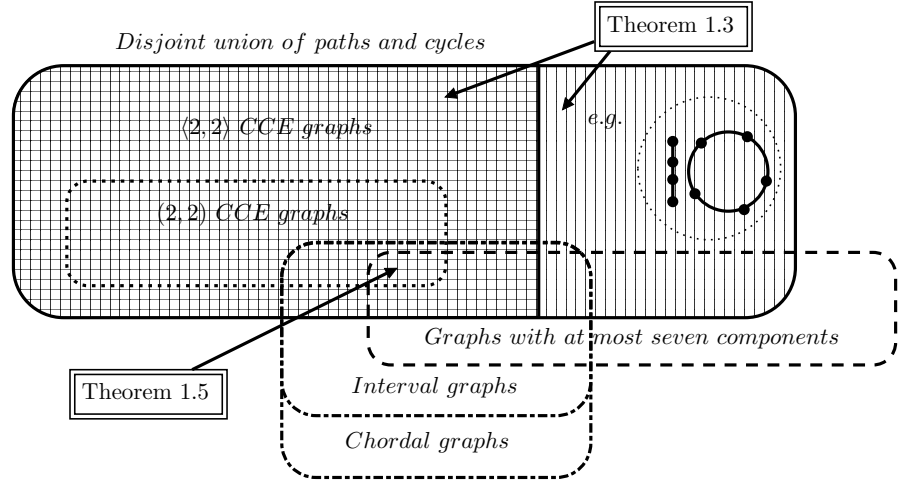
\begin{figure}
\center

  
\tikzset{
pattern size/.store in=\mcSize, 
pattern size = 5pt,
pattern thickness/.store in=\mcThickness, 
pattern thickness = 0.01pt,
pattern radius/.store in=\mcRadius, 
pattern radius = 1pt}
\makeatletter
\pgfutil@ifundefined{pgf@pattern@name@_e69yn3zqk}{
\pgfdeclarepatternformonly[\mcThickness,\mcSize]{_e69yn3zqk}
{\pgfqpoint{-\mcThickness}{-\mcThickness}}
{\pgfpoint{\mcSize}{\mcSize}}
{\pgfpoint{\mcSize}{\mcSize}}
{
\pgfsetcolor{\tikz@pattern@color}
\pgfsetlinewidth{\mcThickness}
\pgfpathmoveto{\pgfpointorigin}
\pgfpathlineto{\pgfpoint{0}{\mcSize}}
\pgfusepath{stroke}
}}
\makeatother

 
\tikzset{
pattern size/.store in=\mcSize, 
pattern size = 5pt,
pattern thickness/.store in=\mcThickness, 
pattern thickness = 0.01pt,
pattern radius/.store in=\mcRadius, 
pattern radius = 1pt}
\makeatletter
\pgfutil@ifundefined{pgf@pattern@name@_ss734l23v lines}{
\pgfdeclarepatternformonly[\mcThickness,\mcSize]{_ss734l23v}
{\pgfqpoint{0pt}{0pt}}
{\pgfpoint{\mcSize+\mcThickness}{\mcSize+\mcThickness}}
{\pgfpoint{\mcSize}{\mcSize}}
{\pgfsetcolor{\tikz@pattern@color}
\pgfsetlinewidth{\mcThickness}
\pgfpathmoveto{\pgfpointorigin}
\pgfpathlineto{\pgfpoint{\mcSize}{0}}
\pgfusepath{stroke}}}
\makeatother
\tikzset{every picture/.style={line width=0.75pt}} 
\scalebox
          {.8}{
\begin{tikzpicture}[x=0.75pt,y=0.75pt,yscale=-1,xscale=1]

\draw  [pattern=_e69yn3zqk,pattern size=3.75pt,pattern thickness=0.1pt,pattern radius=0pt, pattern color={rgb, 255:red, 0; green, 0; blue, 0}][line width=1.5]  (81.33,79.99) .. controls (81.33,62.43) and (95.57,48.19) .. (113.13,48.19) -- (542.53,48.19) .. controls (560.09,48.19) and (574.33,62.43) .. (574.33,79.99) -- (574.33,175.39) .. controls (574.33,192.95) and (560.09,207.19) .. (542.53,207.19) -- (113.13,207.19) .. controls (95.57,207.19) and (81.33,192.95) .. (81.33,175.39) -- cycle ;
\draw  [pattern=_ss734l23v,pattern size=3.75pt,pattern thickness=0.1pt,pattern radius=0pt, pattern color={rgb, 255:red, 0; green, 0; blue, 0}] (113.13,207.19) .. controls (95.57,207.23) and (81.3,193.02) .. (81.26,175.45) -- (81.06,80.05) .. controls (81.03,62.49) and (95.23,48.22) .. (112.8,48.18) -- (412.33,47.56) .. controls (412.33,47.56) and (412.33,47.56) .. (412.33,47.56) -- (412.67,206.56) .. controls (412.67,206.56) and (412.67,206.56) .. (412.67,206.56) -- cycle ;
\draw  [dash pattern={on 1.69pt off 2.76pt}][line width=1.5]  (115,134.2) .. controls (115,126.91) and (120.91,121) .. (128.2,121) -- (340.8,121) .. controls (348.09,121) and (354,126.91) .. (354,134.2) -- (354,173.8) .. controls (354,181.09) and (348.09,187) .. (340.8,187) -- (128.2,187) .. controls (120.91,187) and (115,181.09) .. (115,173.8) -- cycle ;
\draw [line width=1.5]    (497,93) -- (497,135) ;
\draw  [line width=1.5]  (507,115) .. controls (507,101.19) and (518.19,90) .. (532,90) .. controls (545.81,90) and (557,101.19) .. (557,115) .. controls (557,128.81) and (545.81,140) .. (532,140) .. controls (518.19,140) and (507,128.81) .. (507,115) -- cycle ;
\draw  [fill={rgb, 255:red, 0; green, 0; blue, 0 }  ,fill opacity=1 ] (494,96) .. controls (494,94.34) and (495.34,93) .. (497,93) .. controls (498.66,93) and (500,94.34) .. (500,96) .. controls (500,97.66) and (498.66,99) .. (497,99) .. controls (495.34,99) and (494,97.66) .. (494,96) -- cycle ;
\draw  [fill={rgb, 255:red, 0; green, 0; blue, 0 }  ,fill opacity=1 ] (494,134) .. controls (494,132.34) and (495.34,131) .. (497,131) .. controls (498.66,131) and (500,132.34) .. (500,134) .. controls (500,135.66) and (498.66,137) .. (497,137) .. controls (495.34,137) and (494,135.66) .. (494,134) -- cycle ;
\draw  [fill={rgb, 255:red, 0; green, 0; blue, 0 }  ,fill opacity=1 ] (494,109) .. controls (494,107.34) and (495.34,106) .. (497,106) .. controls (498.66,106) and (500,107.34) .. (500,109) .. controls (500,110.66) and (498.66,112) .. (497,112) .. controls (495.34,112) and (494,110.66) .. (494,109) -- cycle ;
\draw  [fill={rgb, 255:red, 0; green, 0; blue, 0 }  ,fill opacity=1 ] (494,120) .. controls (494,118.34) and (495.34,117) .. (497,117) .. controls (498.66,117) and (500,118.34) .. (500,120) .. controls (500,121.66) and (498.66,123) .. (497,123) .. controls (495.34,123) and (494,121.66) .. (494,120) -- cycle ;
\draw  [fill={rgb, 255:red, 0; green, 0; blue, 0 }  ,fill opacity=1 ] (512,97) .. controls (512,95.34) and (513.34,94) .. (515,94) .. controls (516.66,94) and (518,95.34) .. (518,97) .. controls (518,98.66) and (516.66,100) .. (515,100) .. controls (513.34,100) and (512,98.66) .. (512,97) -- cycle ;
\draw  [fill={rgb, 255:red, 0; green, 0; blue, 0 }  ,fill opacity=1 ] (541,93) .. controls (541,91.34) and (542.34,90) .. (544,90) .. controls (545.66,90) and (547,91.34) .. (547,93) .. controls (547,94.66) and (545.66,96) .. (544,96) .. controls (542.34,96) and (541,94.66) .. (541,93) -- cycle ;
\draw  [fill={rgb, 255:red, 0; green, 0; blue, 0 }  ,fill opacity=1 ] (508,129) .. controls (508,127.34) and (509.34,126) .. (511,126) .. controls (512.66,126) and (514,127.34) .. (514,129) .. controls (514,130.66) and (512.66,132) .. (511,132) .. controls (509.34,132) and (508,130.66) .. (508,129) -- cycle ;
\draw  [fill={rgb, 255:red, 0; green, 0; blue, 0 }  ,fill opacity=1 ] (553,121) .. controls (553,119.34) and (554.34,118) .. (556,118) .. controls (557.66,118) and (559,119.34) .. (559,121) .. controls (559,122.66) and (557.66,124) .. (556,124) .. controls (554.34,124) and (553,122.66) .. (553,121) -- cycle ;
\draw  [fill={rgb, 255:red, 0; green, 0; blue, 0 }  ,fill opacity=1 ] (539,139) .. controls (539,137.34) and (540.34,136) .. (542,136) .. controls (543.66,136) and (545,137.34) .. (545,139) .. controls (545,140.66) and (543.66,142) .. (542,142) .. controls (540.34,142) and (539,140.66) .. (539,139) -- cycle ;

\draw  [dash pattern={on 0.84pt off 2.51pt}] (475,113) .. controls (475,88.15) and (495.15,68) .. (520,68) .. controls (544.85,68) and (565,88.15) .. (565,113) .. controls (565,137.85) and (544.85,158) .. (520,158) .. controls (495.15,158) and (475,137.85) .. (475,113) -- cycle ;
\draw  [dash pattern={on 4.5pt off 1.5pt on 2.25pt off 1.5pt}][line width=1.5]  (250.5,181.8) .. controls (250.5,169.76) and (260.26,160) .. (272.3,160) -- (423.55,160) .. controls (435.59,160) and (445.35,169.76) .. (445.35,181.8) -- (445.35,247.2) .. controls (445.35,259.24) and (435.59,269) .. (423.55,269) -- (272.3,269) .. controls (260.26,269) and (250.5,259.24) .. (250.5,247.2) -- cycle ;
\draw  [dash pattern={on 4.5pt off 1.5pt on 2.25pt off 1.5pt}][line width=1.5]  (250.5,181.35) .. controls (250.5,169.38) and (260.2,159.68) .. (272.17,159.68) -- (423.68,159.68) .. controls (435.65,159.68) and (445.35,169.38) .. (445.35,181.35) -- (445.35,286.33) .. controls (445.35,298.3) and (435.65,308) .. (423.68,308) -- (272.17,308) .. controls (260.2,308) and (250.5,298.3) .. (250.5,286.33) -- cycle ;
\draw  [dash pattern={on 5.63pt off 4.5pt}][line width=1.5]  (305,180.2) .. controls (305,172.36) and (311.36,166) .. (319.2,166) -- (623.3,166) .. controls (631.14,166) and (637.5,172.36) .. (637.5,180.2) -- (637.5,222.8) .. controls (637.5,230.64) and (631.14,237) .. (623.3,237) -- (319.2,237) .. controls (311.36,237) and (305,230.64) .. (305,222.8) -- cycle ;
\draw [line width=1.5]    (194,240) -- (319.42,177.78) ;
\draw [shift={(323,176)}, rotate = 153.61] [fill={rgb, 255:red, 0; green, 0; blue, 0 }  ][line width=0.08]  [draw opacity=0] (6.97,-3.35) -- (0,0) -- (6.97,3.35) -- cycle    ;
\draw [line width=1.5]    (457.5,39) -- (430.02,72.89) ;
\draw [shift={(427.5,76)}, rotate = 309.04] [fill={rgb, 255:red, 0; green, 0; blue, 0 }  ][line width=0.08]  [draw opacity=0] (6.97,-3.35) -- (0,0) -- (6.97,3.35) -- cycle    ;
\draw [line width=1.5]    (412.5,48) -- (412.5,207) ;
\draw [line width=1.5]    (447.5,33) -- (357.25,66.6) ;
\draw [shift={(353.5,68)}, rotate = 339.58] [fill={rgb, 255:red, 0; green, 0; blue, 0 }  ][line width=0.08]  [draw opacity=0] (6.97,-3.35) -- (0,0) -- (6.97,3.35) -- cycle    ;

\draw (144,26) node [anchor=north west][inner sep=0.75pt]   [align=left] {$\displaystyle Disjoint\ union\ of\ paths\ and\ cycles$};
\draw (185,79) node [anchor=north west][inner sep=0.75pt]   [align=left] {$\displaystyle \langle 2,2\rangle \ CCE\ graphs$};
\draw (170,133) node [anchor=north west][inner sep=0.75pt]   [align=left] {$\displaystyle ( 2,2) \ CCE\ graphs$};
\draw (442,77) node [anchor=north west][inner sep=0.75pt]   [align=left] {$\displaystyle e.g.$};
\draw (289,278) node [anchor=north west][inner sep=0.75pt]   [align=left] {$\displaystyle Chordal\ graphs$};
\draw (293,243) node [anchor=north west][inner sep=0.75pt]   [align=left] {$\displaystyle Interval\ graphs$};
\draw (339,212) node [anchor=north west][inner sep=0.75pt]   [align=left] {$\displaystyle Graphs\ with\ at\ most\ seven\ components$};
\draw    (120,243) -- (213,243) -- (213,268) -- (120,268) -- cycle (117,240) -- (216,240) -- (216,271) -- (117,271) -- cycle ;
\draw (123,247) node [anchor=north west][inner sep=0.75pt]   [align=left] {Theorem 1.5};
\draw    (451,11) -- (544,11) -- (544,36) -- (451,36) -- cycle (448,8) -- (547,8) -- (547,39) -- (448,39) -- cycle ;
\draw (454,15) node [anchor=north west][inner sep=0.75pt]   [align=left] {Theorem 1.3};
\end{tikzpicture} }%
\caption{Visualization of Theorems~\ref{thm:sumup} and \ref{thm:interval}}
\label{fig:Venn}
\end{figure}

\section{Preliminaries}\label{sect:preliminaries}
\label{sec:Preliminaries}

\begin{proposition} \label{prop:degree}
	A $\langle 2,2 \rangle$ CCE graph has only path components and cycle components.
\end{proposition}
\begin{proof}
It is sufficient to show that the degree of each vertex in a $\langle 2,2 \rangle$ CCE graph is less than or equal to $2$. 
Let $D$ be a $\langle 2,2 \rangle$ digraph.
Take a vertex $v$ of $CCE(D)$.
Since $D$ is a $\langle 2,2 \rangle$ digraph and $CCE(D)$ is a simple graph, the ends of each edge incident to $v$ in $CCE(D)$ have a common prey which is different from a common prey of the ends of another edge incident to $v$.
This implies that $v$ has prey in $D$ at least the number of edges incident to $v$ in $CCE(D)$.
Since $v$ has at most two prey in $D$, the degree of $v$ in $CCE(D)$ is at most $2$.
\end{proof}

The following is an immediate consequence of the definitions of CCE graph and $\langle 2,2 \rangle$ digraph.
\begin{proposition}\label{prop:property_degree2}
Let $D$ be a $\langle 2,2 \rangle$ digraph and $u$ be a vertex which has degree $2$ in $CCE(D)$.
Then the following are true:
\begin{enumerate}[(i)]
	\item $d^+(u)=d^-(u)=2$;
	\item if $v$ is a neighbor of $u$ in $CCE(D)$, then $u$ and $v$ have exactly one common prey and exactly one common predator in $D$.  
\end{enumerate}
\end{proposition}

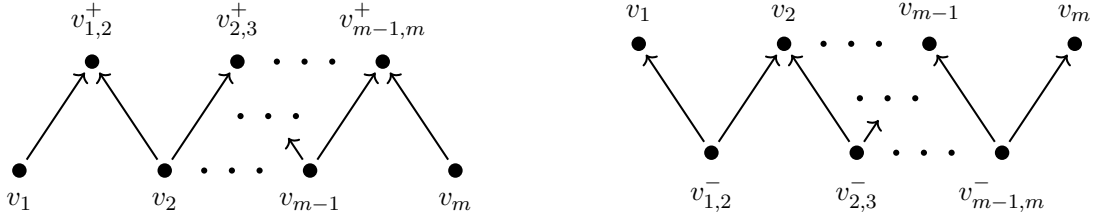
\begin{figure}
	\begin{center}
	{
	\resizebox{0.42\textwidth}{!}{%
	\scalebox{0.8}{
	\begin{tikzpicture}[scale=1]
	\tikzset{mynode/.style={inner sep=2pt,fill,outer sep=2.3pt,circle}}
	\node [mynode] (u1) at (-5,1.5) [label=above:$v^+_{1,2}$] {};
	\node [mynode] (u2) at (-3,1.5) [label=above :$v^+_{2,3}$] {};
	\node [mynode] (u3) at (-1,1.5) [label=above :$v^+_{m-1,m}$] {};
	\node [mynode] (v1) at (-6,0) [label=below :$v_1$] {};
	\node [mynode] (v2) at (-4,0) [label=below :$v_2$] {};
	\node [mynode] (v3) at (-2,0) [label=below :$v_{m-1}$] {};
	\node [mynode] (v4) at (0,0) [label=below :$v_m$] {};
	\draw[->, thick] (v1) edge (u1);
	\draw[->, thick] (v2) edge (u1);
	\draw[->, thick] (v2) edge (u2);
	\draw[->, thick] (v3) edge (u3);
	\draw[->, thick] (v4) edge (u3);
\draw[->, thick] (v3) edge (-2.3,0.45);
	\path (u2) -- (u3) node [font=\Huge, midway] {$\dots$};
	\path (v2) -- (v3) node [font=\Huge, midway] {$\dots$};
	\path (-3.5, 0.75) -- (-1.5, 0.75) node [font=\Huge, midway] {$\dots$};
	\end{tikzpicture}}
	}
	}
	\hspace{1cm}
	{
	\resizebox{0.42\textwidth}{!}{%
	\scalebox{0.8}{
	\begin{tikzpicture}[scale=1]
	\tikzset{mynode/.style={inner sep=2pt,fill,outer sep=2.3pt,circle}}
	\node [mynode] (u1) at (-5,0) [label=below:$v^-_{1,2}$] {};
	\node [mynode] (u2) at (-3,0) [label=below :$v^-_{2,3}$] {};
	\node [mynode] (u3) at (-1,0) [label=below :$v^-_{m-1,m}$] {};
	\node [mynode] (v1) at (-6,1.5) [label=above :$v_1$] {};
	\node [mynode] (v2) at (-4,1.5) [label=above :$v_2$] {};
	\node [mynode] (v3) at (-2,1.5) [label=above :$v_{m-1}$] {};
	\node [mynode] (v4) at (0,1.5) [label=above :$v_m$] {};
	\draw[->, thick] (u1) edge (v1);
	\draw[->, thick] (u1) edge (v2);
	\draw[->, thick] (u2) edge (v2);
	\draw[->, thick] (u3) edge (v3);
	\draw[->, thick] (u3) edge (v4);
    \draw[->, thick] (u2) edge (-2.7,0.45);
	\path (u2) -- (u3) node [font=\Huge, midway] {$\dots$};
	\path (v2) -- (v3) node [font=\Huge, midway] {$\dots$};
	\path (-3.5, 0.75) -- (-1.5, 0.75) node [font=\Huge, midway] {$\dots$};
	\end{tikzpicture}}
	}
	}
	\end{center}
	\caption{Two subdigraphs of $D$ determined by $P_{v,m}$ or $C_{v,m}$ for an integer $m \geq 3$ where $D$ is a $\langle 2, 2\rangle$ digraph whose CCE graph contains $P_{v,m}$ or $C_{v,m}$}
	\label{fig:securing prey_0}
	\end{figure}

\begin{proposition} \label{prop:two_degree_property}
Let $G$ be the $CCE$ graph of a $\langle 2,2 \rangle$ digraph $D$.
If two prey of a vertex are not adjacent in $G$,
then each of them has degree at most one in $G$.
\end{proposition}
\begin{proof}
Let $u$ and $v$ be the two prey of $w$ in $D$.
To show the contrapositive of the statement,
assume that one of $u$ and $v$ has degree at least $2$ in $G$.
Without loss of generality, we may assume that $u$ has degree at least $2$ in $G$.
Then $u$ has degree $2$ in $G$ by Proposition~\ref{prop:degree}.
Since $w \in N^-(u)$, by Proposition~\ref{prop:property_degree2}(ii), $w$ is a common predator of $u$ and one of neighbors of $u$ in $G$.
Thus, since $N^+(w)=\{u,v\}$, $u$ and $v$ are adjacent in $G$.
\end{proof}

Most of graphs considered in this paper are paths or cycles since a $\langle 2,2 \rangle$ CCE graph consist of only path components and cycle components by Proposition~\ref{prop:degree}. 
We denote a path of length $m-1$ and a cycle of length $m$ by $P_m$ and $C_m$, respectively, for a positive integer $m$.
Especially, we denote the path $v_1v_2 \cdots v_m$ and the cycle $v_1v_2 \cdots v_mv_1$ by $P_{v,m}$ and $C_{v,m}$, respectively.
For a given $C_{v,m}$, we identify $v_{m+j}$ with $v_j$ for any integer $j$.

Let $m$ be an integer greater than or equal to $3$.
By Proposition~\ref{prop:property_degree2}, given a cycle $C_{v,m}$,
$v_i$ and $v_{i+1}$ have a unique common prey and a unique common predator for each $1\leq i\leq m$. 
Given a path $P_{v,m}$,
any interior vertex on $P_{v,m}$ has degree $2$ and so $v_i$ and $v_{i+1}$ have a unique common prey and a unique common predator for each $1\leq i\leq m-1$. 
For notational convenience, we denote the unique common prey of $v_i$ and $v_{i+1}$ by $v^+_{i,i+1}$ and the unique common predator of $v_i$ and $v_{i+1}$ by $v^-_{i,i+1}$ when $\{v_i,v_{i+1}\} \subseteq V(P_{v,m})$ (resp.\ $\{v_i,v_{i+1}\} \subseteq V(C_{v,m})$) for each $1\leq i \leq m-1$ (resp.\ $1\leq i \leq m$). See Figure~\ref{fig:securing prey_0} for an illustration.
By the definition of $\langle 2,2 \rangle$ digraph, it is obvious that if  $i\neq j$, then
\[v^-_{i,i+1} \neq v^-_{j,j+1} \quad \text{and} 
\quad  v^+_{i,i+1} \neq v^+_{j,j+1}. \] 

 From now on, we use the notation $u \to v$ to represent ``$(u,v)$ is an arc of a digraph".

\begin{proposition}\label{prop:key1}
	Let $G$ be the CCE graph of a $\langle 2,2 \rangle$ digraph $D$ and $P_{u,\ell}$ and $P_{v,m}$ be two nontrivial paths of $G$ with $\ell \geq 3$.
	Suppose that $u_1 \to v_t$, $u_2 \to v_t$, and $u_2 \to v_{t+1}$ for some integer $t \in \{1, \ldots, m-1\}$.
	Then $u^+_{i,i+1}=v_{t+i-1}$ for each $1\leq i\leq \min (\ell-1,m-t+1)$. 
\end{proposition}
\begin{proof}
By the complete induction on $i$.
By the hypothesis, $u^+_{1,2}=v_{t}$.
Since $u_2 \to v_{t+1}$, $u^+_{2,3}=v_{t+1}$.
Suppose $u^+_{i,i+1}=v_{t+i-1}$ for any $2\leq i \leq k$ with $k<\min (\ell-1,m-t)$.
Since $k < m-t$, $v_{t+k-1}$ is an interior vertex of $P_{v,m}$ and so it has degree $2$.
Therefore, by the contrapositive of Proposition~\ref{prop:two_degree_property},
the other prey of $u_{k+1}$ must be adjacent to $v_{t+k-1}$.
Thus $v_{t+k}$ is a prey of $u_{k+1}$.
Since $k < \ell-1$, $u_{k+1}$ is an interior vertex of $P_{u,\ell}$.
Thus $u_{k+2} \to v_{t+k}$. 
Hence we have shown $u^+_{k+1,k+2}=v_{t+k}$.
\end{proof}

 If $P_{u,\ell}=P_{v,m}$ in Proposition~\ref{prop:key1}, then the condition $v^+_{1,2}=v_t$ implies $v_2 \to v_{t+1}$ as we show in the following useful theorem.

\begin{theorem}\label{thm:key2}
	Let $G$ be the CCE graph of a $\langle 2,2 \rangle$ digraph $D$ and $P_{v,m}$ be a path in $G$ for some integer $m \ge 3$. If $v^+_{1,2}=v_t$ for some integer $t \in \{3, \ldots, m\}$, then  $v^+_{i,i+1}=v_{t+i-1}$ and $v^-_{t+i-2,t+i-1}=v_{i}$ for each integer $1\le i \le m-t+1$ (see Figure~\ref{fig:self-sufficing}).
\end{theorem}
\begin{proof}
	Suppose $v^+_{1,2}=v_t$ for some integer $t \in \{3, \ldots, m\}$.
	Then $N^-(v_t)=\{v_1, v_2\}$ and so $v^-_{t-1,t}$ is either $v_1$ or $v_2$.
	To the contrary, suppose $v^-_{t-1,t}=v_2$.
	Then $v_2\to v_{t-1}$.
	Since $v^+_{1,2}=v_t$, by applying Proposition~\ref{prop:key1} to $P_{v,m}$ and $P^{-1}_{v,m}$ where $P^{-1}_{v,m}$ is the path obtained from $P_{v,m}$ by reversing its sequence, 
	$v^+_{i,i+1}=v_{t-i+1}$ for each $1\leq i \leq t-1$.
%
%
	Now it is easy to check that $v_{\left\lfloor\frac{t+2}{2}\right\rfloor}$ is incident to a loop and we reach a contradiction.
	Therefore $v^-_{t-1,t}=v_1$.
	Then, if $v_{t+1}$ exists, $v^-_{t,t+1}$ must be $v_2$ and so, by applying Proposition~\ref{prop:key1} to $P_{v,m}$ and itself, we reach the desired conclusion.
	\end{proof}
	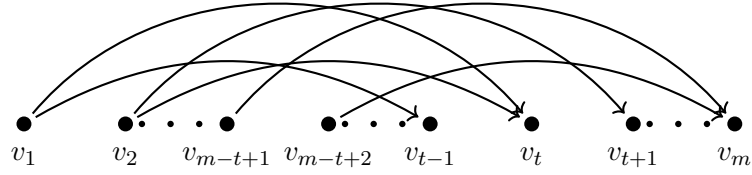
\begin{figure}
			\begin{center}
			\begin{tikzpicture}[scale=0.9]
			\tikzset{mynode/.style={inner sep=2pt,fill,outer sep=2.3pt,circle}}
			\node [mynode] (v1) at (-4.5,0) [label=below :$v_1$] {};
			\node [mynode] (v2) at (-3,0) [label=below :$v_2$] {};
			\node [mynode] (v3) at (-1.5,0) [label=below :$v_{m-t+1}$] {};
			\node [mynode] (v4) at (0,0) [label=below :$v_{m-t+2}$] {};
			\node [mynode] (v5) at (1.5,0) [label=below :$v_{t-1}$] {};
			\node [mynode] (v6) at (3,0) [label=below :$v_{t}$] {};
			\node [mynode] (v7) at (4.5,0) [label=below :$v_{t+1}$] {};
			\node [mynode] (v8) at (6,0) [label=below :$v_{m}$] {};
			\draw[->, thick] (v1) edge[bend left=50] (v6);
			\draw[->, thick] (v2) edge[bend left=50] (v7);
			\draw[->, thick] (v3) edge[bend left=50] (v8);
			\draw[->, thick] (v1) edge[bend left=30] (v5);
			\draw[->, thick] (v2) edge[bend left=30] (v6);
			\draw[->, thick] (v4) edge[bend left=30] (v8);
			\path (v2) -- (v3) node [font=\Huge, midway] {$\dots$};
			\path (v4) -- (v5) node [font=\Huge, midway] {$\dots$};
			\path (v7) -- (v8) node [font=\Huge, midway] {$\dots$};
			\end{tikzpicture}
			\end{center}
			\caption{The arc set in Theorem~\ref{thm:key2} }
			\label{fig:self-sufficing}
			\end{figure}


 The degree boundedness of $\langle 2,2 \rangle$ digraph and the previous proposition ensure the following proposition.
\begin{proposition} \label{prop:component_cycle_neighbor}
	Let $G$ be the CCE graph of a $\langle 2,2 \rangle$ digraph $D$. 
	Given a path $P_{v,m}$ (resp.\ a cycle $C_{v,m}$)  in $G$ for some integer $m \geq 3$, the following are true: 
	\begin{enumerate}[(i)]
\item $v^-_{i,i+1} \neq v^-_{j,j+1}$ and $v^+_{i,i+1} \neq v^+_{j,j+1}$ for distinct $1 \le i,j \le m-1$ (resp.\ $1 \le i,j \le m$); 
\item 
$v^-_{i,i+1}$ is adjacent only to $v^-_{i-1,i}$ or $v^-_{i+1,i+2}$ in $G$ for each $2 \le i \le m-2$ (resp.\ $1\le i \le m$).
\item $v^+_{i,i+1}$ is adjacent only to $v^+_{i-1,i}$ or $v^+_{i+1,i+2}$ in $G$ for each $2 \le i \le m-2$ (resp.\ $1\le i \le m$).
%
\end{enumerate}
\end{proposition}

\section{A proof of Theorem~\ref{thm:sumup}}
\label{sec:Only path components}

This section is devoted to proving Theorem~\ref{thm:sumup} which completely characterizes the $\langle 2,2 \rangle$ CCE graphs.


Given a digraph $D$, we let $D^{\leftarrow}$ be the digraph obtained from $D$ by reversing the direction of each arc in $D$.
Then it is easy to check that $CCE(D)=CCE(D^{\leftarrow})$.

\begin{proposition}\label{prop:reverse_property}
Let $G$ be a $\langle 2,2 \rangle$ CCE graph and $\mathcal{D}$ be the collection of $\langle 2,2 \rangle$ digraphs each of whose CCE graph is $G$.
If the statement $\alpha$ is a property satisfied by each digraph in $\mathcal{D}$,
then the statement $\alpha^{\leftarrow}$ obtained from $\alpha$ by replacing the term `prey' (resp.\ `predator') with the term `predator' (resp.\ `prey') is also satisfied by each digraph in $\mathcal{D}$.
\end{proposition}
\begin{proof}
	Suppose that the statement $\alpha$ is a property satisfied by each digraph in $\mathcal{D}$.
	Take a digraph $D$ in $\mathcal{D}$.
	Then, since $CCE(D)=CCE(D^{\leftarrow})$,
	$D^{\leftarrow} \in \mathcal{D}$ and so $D^{\leftarrow}$ satisfies $\alpha$.
	By the definition of $D^{\leftarrow}$,
	$\alpha^{\leftarrow}$ is satisfied by $D$.
	Since $D$ was arbitrarily chosen from $\mathcal{D}$, $\alpha^{\leftarrow}$ is satisfied by each digraph in $\mathcal{D}$.
\end{proof}

\begin{lemma} \label{lem:a path}
A nontrivial path is not a $\langle 2,2 \rangle$ CCE graph. 
\end{lemma}
\begin{proof}
To the contrary, suppose that a nontrivial path $P_{v,m}$ is the CCE graph of a $\langle 2,2 \rangle$ digraph $D$ for some integer $m>1$.
Any two adjacent vertices of $P_{v,m}$ must have a common prey and a common predator.
Since $D$ is loopless, $m\ge 3$ and $v^+_{1,2}=v_t$ for some integer $t \in \{3, \ldots, m\}$.
By Theorem~\ref{thm:key2}, 
\begin{equation}\label{eq:1}
	v^+_{i,i+1}=v_{t+i-1}
\end{equation}
and 
\begin{equation}\label{eq:2}
v^-_{t+i-2,t+i-1}=v_{i}
\end{equation}
for each $1\le i \le m-t+1$.
Then $v^-_{t-1,t}=v_1$ by substituting $i=1$ in \eqref{eq:2}.
Since $v_t$ is already a prey of $v_2$, the only neighbor $v_2$ of $v_1(=v^-_{t-1,t})$ in $G$ cannot be a predator of $v_{t-1}$.
Since $v_{t-1}$ has degree $2$, there is a predator other than $v_1$ by Proposition~\ref{prop:property_degree2}(i).
Since we have shown that it cannot be $v_2$, it is not adjacent to $v_1$ in $G$.
Thus, by Propositions~\ref{prop:two_degree_property} and \ref{prop:reverse_property}, we have $v_m \to v_{t-1}$ and so
$v^-_{t-2,t-1}=v_m$.
Then $v_{t-1}$ is a common prey of $v_1$ and $v_m$.
We first consider the case where $t<m$.
Then $m-t \ge 1$ and so $v_{m-1}=v^+_{m-t,m-t+1}$ by substituting $i=m-t$ in \eqref{eq:1}.
Thus the only neighbor $v_{m-1}$ of $v_m(=v^+_{m-t+1,m-t+2})$ cannot be a prey of $v_{m-t+2}$.
Note that $v_{m-t+2}$ is an interior vertex of $P_{v,m}$.
By applying a similar argument for $v_{t-1}$, we may show that $v_{m-t+2} \to v_1$ and so
$v^+_{m-t+2,m-t+3}=v_1$.
Thus $v_{m-t+2}$ is a common predator of $v_1$ and $v_m$.
Now we consider the case where $t=m$. 
Then $v^+_{1,2}=v_m$.
Since $v_{t-1}$ is a common prey of $v_1$ and $v_m$, the only neighbor $v_{m-1}(=v_{t-1})$ of $v_m$ cannot be a prey of $v_2$.
By applying a similar argument for $v_{t-1}$, we may show that $v_2 \to v_1$.
Then $v_{2}$ is a common predator of $v_1$ and $v_m$.
Whether $t<m$ or $t=m$, $v_1$ and $v_m$ have a common prey and a common predator in $D$.
Thus $v_1$ and $v_m$ are adjacent in $CCE(D)$, which is a contradiction.
\end{proof}

Given a positive integer $m\ge 3$, we consider a digraph $D_{v,m}^{\stackrel{t}{\curvearrowright}}$ with the vertex set \[V(D_{v,m}^{\stackrel{t}{\curvearrowright}})=\{v_1,v_2, \ldots, v_m\}\] and the arc set 

\[A(D_{v,m}^{\stackrel{t}{\curvearrowright}})=\bigcup_{k=1}^{m}\{(v_k, v_{k+t}),(v_k, v_{k+t+1})\}\]
for some $t \in \{1, \ldots, m-2\}$
(identify $v_{m+i}$ with $v_i$ for each integer $i$).
For each vertex $v_i$ in $D_{v,m}^{\stackrel{t}{\curvearrowright}}$,
\begin{equation}\label{eq:N}
	N^+(v_i)=\{v_{i+t},v_{i+t+1}\} \quad \text{and} \quad  N^-(v_i)=\{v_{i-t-1}, v_{i-t}\}.
\end{equation}
Since $t \in \{1, \ldots, m-2\}$, $D_{v,m}^{\stackrel{t}{\curvearrowright}}$ is loopless and so it is a $\langle  2,2 \rangle$ digraph.
Moreover, $v_{i+t+1}$ (resp.\ $v_{i-t}$) is a common prey (resp.\ predator) of $v_i$ and $v_{i+1}$ for each integer $1\le i \le m$.
Thus $C_{v,m}$ is a subgraph of $CCE(D_{v,m}^{\stackrel{t}{\curvearrowright}})$.
Therefore, by Proposition~\ref{prop:degree},
\begin{equation}\label{eq:CCE(D_{v,m}(t))}
	CCE(D_{v,m}^{\stackrel{t}{\curvearrowright}})=C_{v,m}
\end{equation}
and
\begin{equation}\label{eq:D_{v,m}(t)}
v^+_{i,i+1}=v_{i+t+1}	\quad \text{and} \quad v^-_{i,i+1}=v_{i-t}
\end{equation}
in $D_{v,m}^{\stackrel{t}{\curvearrowright}}$ for each integer $1\le i \le m$.
Hence we obtain the following lemma.

\begin{lemma}\label{lem:a cycle}
		 A cycle of length at least $3$ is a $\langle 2,2 \rangle$ CCE graph.
\end{lemma}
Given two vertex-disjoint graphs $G_1$ and $G_2$,
the {\it union} $G_1 \cup G_2$ of $G_1$ and $G_2$ is the graph with the vertex 
$V(G_1)\cup V(G_2)$ and the edge set $E(G_1)\cup E(G_2)$.
\begin{proposition}\label{prop:Pm,Pn}
		For positive integers $m$ and $n$, $P_m \cup P_n$ is a $\langle 2,2 \rangle$ CCE graph.
	\end{proposition}	
	\begin{proof}
		Fix positive integers $m$ and $n$.
		Since an edgeless graph is a $\langle 2,2 \rangle$ CCE graph, the case $m=n=1$ is clear.
		Without loss of generality, we assume $m \ge 2$.
		We consider the digraph \[D:=D_{v,m+n}^{\stackrel{m-1}{\curvearrowright}}-(v_1,v_{m}).\]
		Since $D_{v,m+n}^{\stackrel{m-1}{\curvearrowright}}$ is a $\langle 2,2 \rangle$ digraph, $D$ is a $\langle 2,2 \rangle$ digraph.
		By \eqref{eq:D_{v,m}(t)}, 
		\[v^+_{i,i+1}=v_{i+m}	\quad \text{and} \quad v^-_{i,i+1}=v_{i-m+1} \]
		in $D_{v,m+n}^{\stackrel{m-1}{\curvearrowright}}$ for each integer $1\le i \le m+n$.
		Especially, $v^+_{m+n,1}=v_{m}$ and $v^-_{m, m+1}=v_1$ in $D_{v,m+n}^{\stackrel{m-1}{\curvearrowright}}$ since we identify $v_{m+n+j}$ with $v_j$ for each integer $j$.
		Thus removing the arc $(v_1,v_{m})$ from $D_{v,m+n}^{\stackrel{m-1}{\curvearrowright}}$ deletes the edges $\{v_1,v_{m+n}\}$ and $\{v_m,v_{m+1}\}$ in $CCE(D_{v,m+n}^{\stackrel{m-1}{\curvearrowright}})$ so that the CCE graph of $D$ is the union of paths \[v_1v_2\cdots v_{m} \quad  \text{and} \quad v_{m+1}v_{m+2} \cdots v_{m+n}\] by \eqref{eq:CCE(D_{v,m}(t))}.
		Hence $CCE(D) \cong P_{m} \cup P_n$. 
	\end{proof}

Given CCE graphs $G_1$ and $G_2$ of $\langle 2,2 \rangle$ digraphs $D_1$ and $D_2$, respectively,
the digraph with the vertex set $V(D_1) \cup V(D_2)$ and the arc set $A(D_1)\cup A(D_2)$ is a $\langle 2,2 \rangle$ digraph whose CCE graph is $G_1 \cup G_2$.
Thus the following lemma is true.

\begin{lemma}\label{lem:union}
	For any $\langle 2,2 \rangle$ CCE graphs $G_1$ and $G_2$, $G_1\cup G_2$ is a $\langle 2,2 \rangle$ CCE graph.
\end{lemma}

We denote $k$ path components $P_m$ of a graph by $kP_m$ for positive integers $k \geq 2$ and $m$.

\begin{proposition}\label{prop:Pm,Pm,Pn}
	For positive integers $m$ and $n$, $2P_m \cup P_n$ is a $\langle 2,2 \rangle$ CCE graph.
\end{proposition}
	\begin{proof}
		Fix positive integers $m$ and $n$.
		Since $P_n \cup P_1$ is the CCE graph of a $\langle 2,2 \rangle$ digraph $D$ by Proposition~\ref{prop:Pm,Pn}, $P_n \cup 2P_1$ is the CCE graph of the $\langle 2,2 \rangle$ digraph by Lemma~\ref{lem:union}.
		Now we assume $m \ge 2$.
		We consider the digraph \[D:=D_{v,2m+n}^{\stackrel{m-1}{\curvearrowright}}-(v_1,v_{m})-(v_{m+1},v_{2m}).\]
		Since $D_{v,2m+n}^{\stackrel{m-1}{\curvearrowright}}$ is a $\langle 2,2 \rangle$ digraph, $D$ is a $\langle 2,2 \rangle$ digraph.
		Note that by \eqref{eq:D_{v,m}(t)}, 
		\[v^+_{i,i+1}=v_{i+m}	\quad \text{and} \quad v^-_{i,i+1}=v_{i-m+1} \]
		in $D_{v,2m+n}^{\stackrel{m-1}{\curvearrowright}}$ for each integer $1\le i \le 2m+n$.
		Especially, $v^+_{2m+n,1}=v_{m}$, $v^-_{m, m+1}=v_1$, $v^+_{m,m+1}=v_{2m}$, and $v^-_{2m, 2m+1}=v_{m+1}$ in $D_{v,2m+n}^{\stackrel{m-1}{\curvearrowright}}$.
		Thus removing the arcs $(v_1,v_{m})$ and $(v_{m+1},v_{2m})$ from $D_{v,2m+n}^{\stackrel{m-1}{\curvearrowright}}$ deletes the edges $\{v_1,v_{2m+n}\}$, $\{v_m,v_{m+1}\}$, and $\{v_{2m},v_{2m+1}\}$ so that the CCE graph of $D$ is the union of paths 
		\[v_1v_2\cdots v_{m},\quad v_{m+1}v_{m+2} \cdots v_{2m}, \quad \text{and} \quad v_{2m+1}v_{2m+2}\cdots v_{2m+n}\]
		by \eqref{eq:CCE(D_{v,m}(t))}.
		Hence $CCE(D) \cong 2P_m \cup P_n$. 
	\end{proof}

\begin{proposition}\label{prop:Pl,Pm,Pn}
	For positive integers $l$, $m$, and $n$, $P_l \cup P_m \cup P_n$ is a $\langle 2,2 \rangle$ CCE graph.
	\end{proposition}

\begin{proof}
Fix positive integers $l$, $m$, and $n$.
Without loss of generality, we may assume $l\leq m \leq n$.
If $l=m$ or $m=n$, then $P_l \cup P_m \cup P_n$ is a $\langle 2,2 \rangle$ CCE graph by
Proposition~\ref{prop:Pm,Pm,Pn}.
Now suppose $l < m <n$.
If $l=1$,
then $P_1 \cup P_m \cup P_n$ is a $\langle 2,2 \rangle$ CCE graph by Propositions~\ref{prop:Pm,Pn} and Lemma~\ref{lem:union}.
Thus it remains to consider the case \[1<l<m<n.\] 
Let $D_1= D_{u,l+m}^{\stackrel{l-1}{\curvearrowright}}$ and $D_2=D_{v,m+n}^{\stackrel{m-l}{\curvearrowright}}$.
We consider two digraphs \[D_3:=D_1-\{(u_{l+i}, u_{2l-1+i})\colon\, 1\le i \le m-l+1 \}-\{(u_{l+i}, u_{2l+i})\colon\, 1\le i \le m-l \} \]
and \[D_4:=D_2-\{(v_{n+i}, v_{n+m-l+i})\colon\, 1\le i \le l \}-\{(v_{n+i}, v_{n+m-l+1+i})\colon\, 1\le i \le l-1 \}\]
(see Figure~\ref{fig:prop:Pl,Pm,Pn}).
Now we obtain a digraph $D$ from digraphs $D_3$ and $D_4$ by identifying $u_{l+i}$ with $v_{m+n+1-i}$ for each $1 \le i \le m$.
Since $D_1$ and $D_2$ are $\langle 2,2 \rangle$ digraphs, $D_3$ and $D_4$ are $\langle 2,2 \rangle$ digraphs.
To show that $D$ is a $\langle 2,2 \rangle$ digraph, it suffices to check the outdegree and indegree of the vertices identified in $D_3$ and $D_4$.
By \eqref{eq:N}, we may check the following:
\begin{itemize}
	\item $d^-_{D_3}(u_{l+i})=2$ for $1\leq i\leq l-1$;
	$d^-_{D_3}(u_{2l})=1$; 
	$d^-_{D_3}(u_{2l+i})=0$ for $1\leq i \leq m-l$;
	\item $d^+_{D_3}(u_{l+i})=0$ for $1\leq i \leq m-l$; $d^+_{D_3}(u_{m+1})=1$, $d^+_{D_3}(u_{m+i+1})=2$ for $1 \leq i \leq l-1$; 
	\item $d^-_{D_4}(v_{n+i})=2$ for $1\leq i \leq m-l$; $d^-_{D_4}(v_{m+n-l+1})=1$; $d^-_{D_4}(v_{m+n-l+i+1})=0$ for $1\leq i \leq l-1$;
	\item $d^+_{D_4}(v_{n+i})=0$ for $1\leq i \leq l-1$; $d^+_{D_4}(v_{n+l})=1$; $d^+_{D_4}(v_{n+l+i})=2$ for $1\leq i \leq m-l$
\end{itemize}
Thus each vertex identified in $D_3$ and $D_4$ has outdegree $2$ and indegree $2$ in $D$.
Therefore $D$ is a $\langle 2,2 \rangle$ digraph.
By \eqref{eq:D_{v,m}(t)}, we may check the following:
\begin{itemize}
	\item $A(D_3)$ guarantees that $u_{m+i,m+1+i}^+=u_i$ and $u_{i,i+1}^-=u_{m+1+i}$ for $1\le i \le 2l-1$;
	\item $A(D_4)$ guarantees that $v_{n+l-1+i,n+l+i}^+=v_i$ and $v_{i,i+1}^-=v_{n+1+i}$ for $1\le i \le m+n-l$.
\end{itemize} 
Since we have identified $u_{l+i}$ with $v_{m+n+1-i}$ for each $1 \le i \le m$, $CCE(D)$ is the union of paths $u_1\cdots u_l$, $u_{l+1}\cdots u_{l+m}$, and $v_1\cdots v_n$.\end{proof}

\begin{lemma}\label{lem:only path components}
	Let $G$ be a disjoint union of paths.
	Then $G$ is a $\langle 2,2 \rangle$ CCE graph if and only if $G$ is not isomorphic to a nontrivial path.
\end{lemma}
\begin{proof}
The ``only if" part follows from Lemma~\ref{lem:a path}.
Suppose that $G$ is not isomorphic to a nontrivial path.
If $G$ is a trivial graph, then it is clear.
Assume that $G$ is not a trivial graph.
Then $G$ has at least $2$ path components and so the path components can be grouped into sets of size $2$ or $3$. 
If one of the grouped sets has a size of $2$,
then the two path components in it is a $\langle 2,2 \rangle$ CCE graph by Proposition~\ref{prop:Pm,Pn}.
Otherwise, the three path components in it is a $\langle 2,2 \rangle$ CCE graph by Proposition~\ref{prop:Pl,Pm,Pn}.
Thus, by Lemma~\ref{lem:union},
$G$ is a $\langle 2,2 \rangle$ CCE graph.
Therefore we have shown the ``if" part.
\end{proof}
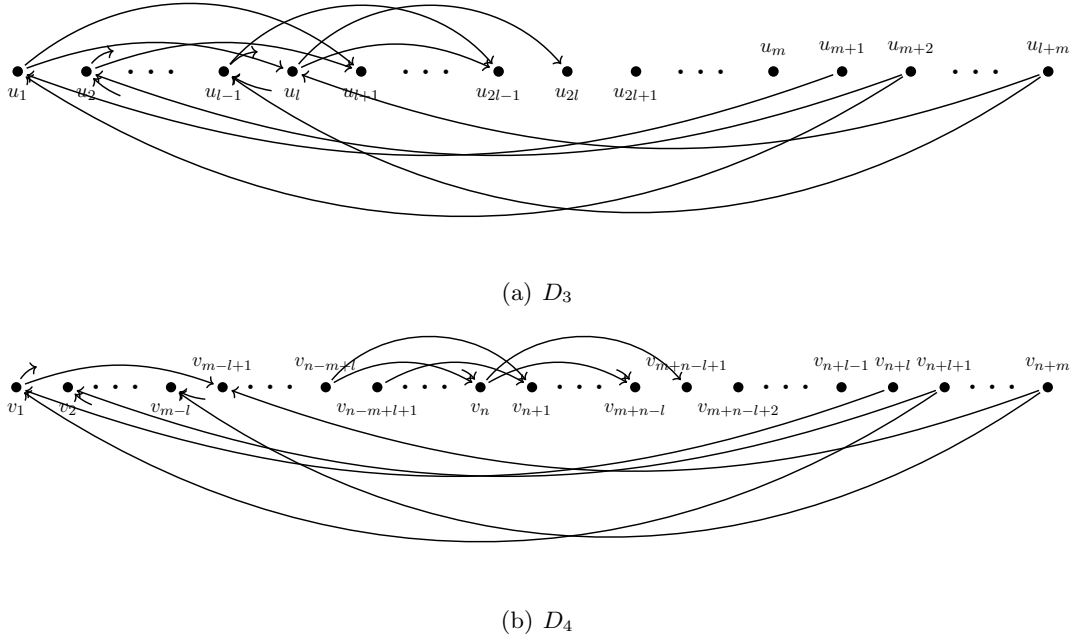
\begin{figure}[t]
			\begin{center}
			\subfigure[$D_3$]{
			\resizebox{0.9\textwidth}{!}{%
			\begin{tikzpicture}[scale=1.3]
			\tikzset{mynode/.style={inner sep=2pt,fill,outer sep=2.3pt,circle}}
			\node [mynode] (u1) at (-6,0) [label=below :$u_1$] {};
			\node [mynode] (u2) at (-5,0) [label=below :$u_2$] {};
			\node [mynode] (ul-1) at (-3,0) [label=below :$u_{l-1}$] {};
			\node [mynode] (ul) at (-2,0) [label=below :$u_{l}$] {};
			\node [mynode] (ul+1) at (-1,0) [label=below :$u_{l+1}$] {};
			\node [mynode] (u2l-1) at (1,0) [label=below :$u_{2l-1}$] {};
			\node [mynode] (u2l) at (2,0) [label=below :$u_{2l}$] {};
			\node [mynode] (u2l+1) at (3,0) [label=below :$u_{2l+1}$] {};
			\node [mynode] (ul+m) at (9,0) [label=above :$u_{l+m}$] {};
			\node (uu) at (-2.4,0.3){};
			\node (uuu) at (-4.4,-0.4){};
			\node (uuuu) at (-2.2,-0.3){};

			\node [mynode] (um) at (5,0) [label=above :$u_{m}$] {};
			\node [mynode] (um+1) at (6,0) [label=above :$u_{m+1}$] {};
			\node [mynode] (um+2) at (7,0) [label=above :$u_{m+2}$] {};
			\draw[->, thick] (u1) edge[bend left=20] (ul);
			\draw[->, thick] (u1) edge[bend left=40] (ul+1);
			\draw[->, thick] (u2) edge[bend left=20] (ul+1);
			\draw[->, thick] (u2) edge[bend left=20] (-4.6,0.3);
			\draw[->, thick] (ul-1) edge[bend left=50] (u2l-1);
			\draw[->, thick] (ul) edge[bend left=50] (u2l);
			\draw[->, thick] (ul) edge[bend left=25] (u2l-1);
			\draw[->, thick] (ul-1) edge[bend left=20] (uu);
			
			\draw[->, thick] (ul+m) edge[bend left=20] (ul);
			\draw[->, thick] (ul+m) edge[bend left=35] (ul-1);
			\draw[->, thick] (um+2) edge[bend left=20] (u2);
			\draw[->, thick] (um+2) edge[bend left=33] (u1);
			\draw[->, thick] (um+1) edge[bend left=20] (u1);
			\draw[->, thick] (uuu) edge[bend left=10] (u2);
			\draw[->, thick] (uuuu) edge[bend left=10] (ul-1);
			\path (u2) -- (ul-1) node [font=\Huge, midway] {$\dots$};
			\path (ul+1) -- (u2l-1) node [font=\Huge, midway] {$\dots$};
			\path (u2l+1) -- (um) node [font=\Huge, midway] {$\dots$};		
			\path (um+2) -- (ul+m) node [font=\Huge, midway] {$\dots$};

			\end{tikzpicture}
			}}
			\subfigure[$D_4$]{
			\resizebox{0.9\textwidth}{!}{%
			\begin{tikzpicture}[scale=1]
			\tikzset{mynode/.style={inner sep=2pt,fill,outer sep=2.3pt,circle}}
			\node [mynode] (v1) at (-9,0) [label=below :$v_{1}$] {};
			\node [mynode] (v2) at (-8,0) [label=below :$v_{2}$] {};
			\node [mynode] (vm-l) at (-6,0) [label=below :$v_{m-l}$] {};
			\node [mynode] (vm-l+1) at (-5,0) [label=above :$v_{m-l+1}$] {};
			\node [mynode] (vn-m+l) at (-3,0) [label=above :$v_{n-m+l}$] {};
			\node [mynode] (vn-m+l+1) at (-2,0) [label=below :$v_{n-m+l+1}$] {};
			\node [mynode] (vn) at (0,0) [label=below :$v_{n}$] {};
			\node [mynode] (vn+1) at (1,0) [label=below :$v_{n+1}$] {};
			\node [mynode] (vn+m-l) at (3,0) [label=below :$v_{m+n-l}$] {};
			\node [mynode] (vn+m-l+1) at (4,0) [label=above :$v_{m+n-l+1}$] {};
			\node [mynode] (vn+m-l+2) at (5,0) [label=below :$v_{m+n-l+2}$] {};
			\node [mynode] (vn+l-1) at (7,0) [label=above :$v_{n+l-1}$] {};
			\node [mynode] (vn+l) at (8,0) [label=above :$v_{n+l}$] {};
			\node [mynode] (vn+l+1) at (9,0) [label=above :$v_{n+l+1}$] {};
			\node [mynode] (vn+m) at (11,0) [label=above :$v_{n+m}$] {};

			\draw[->, thick] (v1) edge[bend left=18] (-8.6,0.4);
			\draw[->, thick] (v1) edge[bend left=20] (vm-l+1);
			\draw[->, thick] (vn-m+l) edge[bend left=30] (vn);
			\draw[->, thick] (vn-m+l) edge[bend left=50] (vn+1);
			\draw[->, thick] (vn-m+l+1) edge[bend left=30] (vn+1);
			\draw[->, thick] (vn) edge[bend left=50] (vn+m-l+1);
			\draw[->, thick] (vn) edge[bend left=30] (vn+m-l);
			\draw[->, thick] (vn+m) edge[bend left=20] (vm-l+1);
			\draw[->, thick] (vn+m) edge[bend left=35] (vm-l);
			\draw[->, thick] (vn+l+1) edge[bend left=34] (v1);
			\draw[->, thick] (vn+l+1) edge[bend left=20] (v2);
			\draw[->, thick] (vn+l) edge[bend left=20] (v1);
			\node (v) at (-7.4,-0.4){};
			\draw[->, thick] (v) edge[bend left=10] (v2);
			\node (vv) at (-5.2,-0.25){};
			\draw[->, thick] (vv) edge[bend left=10] (vm-l);
			\node (vvv) at (-0.5,0.4){};
			\draw[->, thick] (vvv) edge[bend left=15] (vn);	
			\node (vvvv) at (2.5,0.4){};
			\draw[->, thick] (vvvv) edge[bend left=15] (vn+m-l);

			\path (v2) -- (vm-l) node [font=\Huge, midway] {$\dots$};
			\path (vm-l+1) -- (vn-m+l) node [font=\Huge, midway] {$\dots$};
			\path (vn-m+l+1) -- (vn) node [font=\Huge, midway] {$\dots$};
			\path (vn+1) -- (vn+m-l) node [font=\Huge, midway] {$\dots$};
			\path (vn+m-l+2) -- (vn+l-1) node [font=\Huge, midway] {$\dots$};
			\path (vn+l+1) -- (vn+m) node [font=\Huge, midway] {$\dots$};
			\end{tikzpicture}
			}}
			\end{center}
			\caption{Digraphs in the proof of Proposition~\ref{prop:Pl,Pm,Pn}}
			\label{fig:prop:Pl,Pm,Pn}
			\end{figure}

Now we are ready to prove Theorem~\ref{thm:sumup}.

\begin{proof}[Proof of Theorem~\ref{thm:sumup}]
(b)$\Rightarrow$(a) is true by Lemmas~\ref{lem:a cycle},~\ref{lem:union}, and~\ref{lem:only path components}.

To show (a)$\Rightarrow$(b), suppose that $G$ is the CCE graph of a $\langle 2,2 \rangle$ digraph $D$.
	By Proposition~\ref{prop:degree}, $G$ is a disjoint union of paths and cycles.
	To the contrary, suppose that $P_{u,l}$ is the unique path component of $G$ for an integer $l\ge 2$.
	Since a nontrivial path cannot be the CCE graph of a $\langle 2,2 \rangle$ digraph by Lemma~\ref{lem:a path}, some consecutive vertices of $P_{u,l}$ have a common prey or a common predator on a component $C$ of $G$ except $P_{u,l}$.
	Thus $C=C_{v,m}$ for some positive integer $m\ge 3$ since $P_{u,l}$ is the unique path component of $G$.
	If some consecutive vertices of $P_{u,l}$ have a common prey  on $C_{v,m}$, then 	
	they have a common predator on $C_{v,m}$ in $D^{\leftarrow}$ and so it suffices to consider $D^{\leftarrow}$ instead of $D$. 
	Without loss of generality, we may assume that some consecutive vertices of $P_{u,l}$ have a common predator on $C_{v,m}$.  
		That is, $v_j \to u_i$ and $v_j \to u_{i+1}$ for some $i \in \{1, \ldots, l-1\}$ and $j \in \{1, \ldots, m\}$.
Without loss of generality, we may assume that $i$ is the minimum among such the indices. 
Since we may cyclically permute the sequence of $C_{v,m}$ to have a cycle isomorphic to $C_{v,m}$, we may assume $j=2$ and $v_1 \to u_i$.
If $i > 1$, then $v_{1} \to u_{i-1}$ since $u_i$ is a common prey of $v_{1}$ and $v_2$ and $v_2 \to u_{i+1}$, which contradicts the choice of $i$.
Thus $i=1$.
Then 
\[ v_1 \to u_1,\quad v_2 \to u_1, \quad \text{and} \quad v_2 \to u_2.\]
Now $v_{1}$ has a prey $w$ other than $u_1$ and, by Proposition~\ref{prop:two_degree_property}, $w$ must be an end vertex on some path component in $G$.
 Since $P_{u,l}$ is the only path component in $G$,
  $w$ is an end vertex of $P_{u,l}$ and so $w=u_l$.
That is,\[
 	v_1 \to u_l
\]
If $l=2$, then $N^+(v_1)=N^+(v_2)$, which contradicts Proposition~\ref{prop:property_degree2}(ii).
Thus \[ l \geq 3.\]
  Then, by Proposition~\ref{prop:key1},
\begin{equation}\label{eq:lem:only path components_1}
v^+_{i,i+1}=u_{i}
\end{equation} for each $1\leq i \leq \min(l-1,m)$.
To show $l\leq m$,
suppose, to the contrary, that $l > m$.
 Then, by \eqref{eq:lem:only path components_1}, $v^+_{m-2,m-1}=u_{m-2}$ and $v^+_{m-1,m}=u_{m-1}$. 
 Then $v_m$ is the only possible common predator of $u_{m-1}$ and $u_m$.
 Thus $v_m \to u_m$ and so $v_m=u^-_{m-1,m}$.
 However, since $N^+(v_1)=\{u_1,u_l\}$, $v_1$ and $v_m$ have no common prey and so $v_1$ and $v_m$ cannot be adjacent in $G$, which is a contradiction.
 Therefore $l\leq m$.
 Then, by \eqref{eq:lem:only path components_1},
 $v^+_{l-2,l-1}=u_{l-2}$ and so 
 $v_{l-1}=u^-_{l-2,l-1}$.
 Thus $v_l$ is the only possible common predator of $u_{l-1}$ and $u_l$ and so $v_l \to u_l$.
 Then $v^+_{l,l+1}=u_l$.
 Therefore any vertex on $P_{u,l}$
 has two predators on $C_{v,m}$, which implies that two consecutive vertices of $P_{u,l}$ requiring a common prey cannot have it on $P_{u,l}$. 
 Hence some consecutive vertices of $P_{u,l}$ have a common prey on a cycle component.
 Noting that $CCE(D^{\leftarrow})=CCE(D)$, we may also conclude that some consecutive vertices of $P_{u,l}$ have a common predator on a cycle component $C'$ in $CCE(D^{\leftarrow})$.
Then the situation has become identical to the one discussed earlier for $D$. 
Therefore the end vertices $u_1$ and $u_l$ of $P_{u,l}$ have a common predator in $D^{\leftarrow}$ lying on the cycle $C'$ in $CCE(D^{\leftarrow})$. 
Then $u_1$ and $u_l$ have a common prey in $D$.
%
%
%
%
%
Consequently, we have shown that $u_1$ and $u_l$ have a common prey and a common predator in $D$. 
Hence $u_1$ and $u_l$ are adjacent in $G$, which contradicts the fact that $P_{u,l}$ is a path in $G$.
\end{proof}

\section{A proof of Theorem~\ref{thm:interval}}\label{sect:(2,2)}
Theorem~\ref{thm:interval} specifically pertains to the CCE graph of an acyclic  $\langle 2,2 \rangle$ digraph, so in this section, we will focus solely on (2,2) digraphs.
Since the CCE graphs of $(2,2)$ digraphs with at most two vertices are empty graphs, 
Theorem~\ref{thm:interval} is trivially true for $(2,2)$ digraphs with at most two vertices.
In this vein, we only consider $(2,2)$ digraphs with at least three vertices unless otherwise mentioned.

We call a vertex of indegree $0$ (resp.\ outdegree $0$) in a digraph $D$ a {\em source} (resp.\ {\em sink}) of $D$.
It is a well-known fact that
if a digraph $D$ is acyclic, then $D$ has a sink and a source.
Each sink and each source of a digraph form isolated vertices in its CCE graph.

Given a family $\mathcal{D}$ of digraphs, we say that a digraph $D$ in $\mathcal{D}$ is {\it minimal} in $\mathcal{D}$ if there is no proper subdigraph $D'$ of $D$ in $\mathcal{D}$ such that $CCE(D)=CCE(D')$. 
By the Well-Ordering Axiom, the following proposition is true.

\begin{proposition}\label{prop:existence_minimal}
For a $(2,2)$ CCE graph $G$ and the set $\mathcal{D}_G$ of $(2,2)$ digraphs whose CCE graphs are $G$,
there exists a minimal digraph in $\mathcal{D}_G$.
\end{proposition}

Given a $(2,2)$ CCE graph $G$,
we say that
a digraph $D$ is a {\it minimal digraph of $G$}
if $D$ is a minimal digraph among the $(2,2)$ digraphs whose CCE graphs are $G$.
By Proposition~\ref{prop:existence_minimal}, for any $(2,2)$ CCE graph $G$, there exists a minimal digraph of $G$.

It is easy to check that if $D$ is a minimal digraph of a $(2,2)$ CCE graph $G$, then   
$D^{\leftarrow}$ is also a minimal digraph of $G$.
Therefore the proof of Proposition~\ref{prop:reverse_property} can be directly applied to establish the following result.

\begin{proposition}\label{prop:reverse_property_minimal}
Let $G$ be a $( 2,2 )$ CCE graph and $\mathcal{D}$ be the collection of minimal digraphs each of whose CCE graph is $G$.
If the statement $\alpha$ is a property satisfied by each digraph in $\mathcal{D}$,
then the statement $\alpha^{\leftarrow}$ obtained from $\alpha$ by replacing the term `prey' (resp.\ `predator') with the term `predator' (resp.\ `prey') is also satisfied by each digraph in $\mathcal{D}$.
\end{proposition}

\begin{proposition} \label{prop:minimal_property}
Let $D$ be a minimal digraph of a $(2,2)$ CCE graph $G$.
Then the following are true:
\begin{enumerate} [(i)]
\item if a vertex $v$ has exactly one predator (resp.\ one prey), then $v$ has degree $1$ in $G$ and
    the predator (resp.\ the prey) of $v$ has the other prey (resp.\ predator) that is adjacent to $v$ in $G$;
\item if a vertex $v$ has two predators (resp.\ two prey), then $v$ has degree $2$ or the predators (resp.\ the prey) of $v$ are adjacent in $G$;
\item any two distinct vertices have at most one common prey and at most one common predator. 

\end{enumerate}
\end{proposition}

\begin{proof}
By Proposition~\ref{prop:reverse_property_minimal}, for showing (i) and (ii), it is sufficient to handle the case where a vertex $v$ has indegree $1$ or $2$.

To show part (i),
suppose that $v$ has indegree $1$ in $D$.
Then, since $D$ is a $(2,2)$ digraph,
$v$ has degree at most $1$ in $G$.
Suppose that $v$ has degree $0$ in $G$.
Then $CCE(D') = G$ for
the subdigraph $D'$ with $V(D')=V(D)$ and $A(D') \subsetneq A(D)$ obtained from deleting the incoming arc to $v$, which contradicts the minimality of $D$.
Therefore $v$ has degree $1$ in $G$.
Thus the predator of $v$ has the other prey that is adjacent to $v$ in $G$.

To verify part (ii),
we suppose that $v$ has indegree $2$ in $D$.
Let $w$ and $x$ be the predators of $v$.
Assume that $v$ has degree at most $1$, and $w$ and $x$ are not adjacent in $G$.
Then deleting any arc of $(w,v)$ and $(x,v)$ does not change the adjacency between $w$ and $x$.
Moreover, since $v$ has degree at most $1$ and $D$ is a $(2,2)$ digraph,
we may delete one arc of $(w,v)$ and $(x,v)$ so that
the degree of $v$ stays the same in the CCE graph of the resulting digraph $D'$.
Thus $A(D') \subsetneq A(D)$ and $CCE(D')=G$, which contradicts the minimality of $D$.
Hence $v$ has degree $2$ or $w$ and $x$ are adjacent in $G$.

To show part (iii),
suppose to the contrary that there are two distinct vertices $u_1$ and $u_2$ such that they have at least two common prey or at least two common predators.
If $u_1$ and $u_2$ have two common predators, then the predators have $u_1$ and $u_2$ as common prey.
Thus we may assume that
 $u_1$ and $u_2$ have at least two common prey $v_1$ and $v_2$.
 Since $D$ is a $(2,2)$ digraph,
 \[N^+_D(u_1)=N^+_D(u_2)=\{v_1,v_2\}\]
 and 
  \[N^-_D(v_1)=N^-_D(v_2)=\{u_1,u_2\}.\]
  Then the pairs that may be affected by deleting the arc $(u_1,v_1)$ from $D$ are $\{u_1, u_2\}$ and $\{v_1, v_2\}$.
  Yet, the adjacency of $u_1$ and $u_2$ is preserved by the arcs $(u_1,v_2)$ and $(u_2,v_2)$ and that of $v_1$ and $v_2$ is preserved by the arcs $(u_2,v_1)$ and $(u_2,v_2)$.
  Therefore the CCE graph of the digraph $D-(u_1,v_1)$ is isomorphic to $G$, which contradicts the fact that $D$ is minimal.
\end{proof}

\begin{remark}
By Proposition~\ref{prop:minimal_property}(iii),
two adjacent vertices in the $(2,2)$ CCE graph $G$ of a minimal digraph have a unique common predator and a unique common prey.
Therefore the notations $v^-_{1,2}$ and $v^+_{1,2}$ may be used for the common predator and the common prey of those two vertices on a component $P_{v,2}$ of $G$.
\end{remark}

\begin{proposition}\label{prop:(2,2)}
	Let $G$ be the CCE graph of a $(2,2)$ digraph $D$.
	If there is a cycle in $G$, then there is no arc between its vertices.
\end{proposition}
\begin{proof}
To the contrary, suppose that there is a cycle $C_{v,m}$ in $G$ and an arc $(v_i,v_j)$ exists in $D$ for some $m\ge 3$ and $i,j \in \{1,\ldots,m\}$.
Since each vertex on $C_{v,m}$ has degree $2$, $v_i$ is a common predator of $v_j$ and one of $v_{j-1}$ and $v_{j+1}$.
By symmetry, we may assume $j=1$, and $v_1$ and $v_2$ are prey of $v_i$. 
Then $2<i \le m$.
	By applying Theorem~\ref{thm:key2} to $C_{v,m}-v_1v_m$, we have $v_i$, $\cdots$, $v_m$ as common prey of $v_1$ and $v_2$, $\cdots$, $v_{m-i+1}$ and $v_{m-i+2}$, respectively.
	By Theorem~\ref{thm:key2} applied to the path $C_{v,m}-v_1v_2$,  $v_1$ is a common prey of $v_{m-i+2}$ and $v_{m-i+3}$.
	By applying the same theorem to $C_{v,m}-v_2v_3, \ldots, C_{v,m}-v_{i-1}v_i$ repeatedly, we may obtain an arc set of $D$
	\[A:= \bigcup_{k=1}^{m}\{(v_k, v_{k+i-1}),(v_{k+1}, v_{k+i-1})\}.\]
	We consider the subgraph $D'$ of $D$ induced by $A$.
	Then each vertex of $D'$ has outdegree $2$ and so $D'$ has no sink.
	Therefore $D'$ is not acyclic, which contradicts that $D$ is acyclic. 
\end{proof}

Given a vertex set $X$ of a digraph $D$, we denote by $N^+(X)$ and $N^-(X)$ the sets \[\{v \in V(D) \mid (x,v) \in A(D), \ x\in X, \ v \notin X\} \quad \text{and} \quad \{v \in V(D) \mid (v,x) \in A(D),\ x\in X, \ v \notin X\},\]respectively.

\begin{lemma} \label{lem:cycle-char}
Let $G$ be the $CCE$ graph of a $(2,2)$ digraph $D$. Suppose that $G$ has a cycle $C$ of length $m$ for some $m \geq 3$.
Then $|N^+(C) \cup N^-(C) | \geq m+3$
   and $|N^+(C) \cap N^-(C) | \leq m-3$.
\end{lemma}
\begin{proof}
 Let $C:=C_{v,m}$.
 By Proposition~\ref{prop:component_cycle_neighbor}(i) and Proposition~\ref{prop:(2,2)},
 \[N^+(C)=\{v^+_{1,2},v^+_{2,3},\ldots,v^+_{m,1}\} \text{ and } N^-(C)=\{v^-_{1,2},v^-_{2,3},\ldots,v^-_{m,1}\}\]
 with
 \[
|N^+(C)|=|N^-(C)|=m.
\]
Suppose, to the contrary, that
$|N^+(C) \cup N^-(C) | \leq m+2$.
Then, as we have shown $| N^+(C)|=|N^-(C)|=m$, 
\[|N^+(C)\cap N^-(C)| \geq m-2
\]
and so \begin{equation}\label{eq:lem:cycle-char-01}|N^+(C)- N^-(C) |\leq 2.
 \end{equation}	
Take a vertex $x_1$ in $N^+(C)\cap N^-(C)$.
Then $x_1=v^-_{j,j+1}$ for some $j \in \{1,\ldots,m\}$.
Let $u_i=v^+_{i,i+1}$ for each $1\le i \le m$ ($u_{t+m}=u_t$ for any positive integer $t$).
Since $\{u_{j-1},u_{j},u_{j+1}\} \subseteq N^+(C)$, 
at least one of $u_{j-1},u_{j},u_{j+1}$ belongs to $N^+(C)\cap N^-(C)$ by \eqref{eq:lem:cycle-char-01}.
Let $x_2$ be one of such vertices.
Then,
since $v_j\to u_{j-1}$, $v_j \to u_{j}$, and $v_{j+1} \to u_{j+1}$,
we obtain an $(x_1,x_2)$-directed walk $W_1$.
By similar argument,
we obtain an $(x_2,x_3)$-directed walk $W_2$ for some $x_3 \in N^+(C)\cap N^-(C)$.
By repeating this process,
we obtain the directed walk $W:=W_1 \to W_2 \to \cdots \to W_m$ where $W_i$ is an $(x_i,x_{i+1})$-directed walk and $x_i \in N^+(C) \cap N^-(C)$ for each $1\le i \le m$.
Then each of $x_1,\ldots,x_{m+1}$ belongs to $N^+(C)\cap N^-(C)$.
By the way, since $|N^+(C)\cap N^-(C)| \leq m$, 
$x_k=x_\ell$ for some distinct $k,\ell \in \{1, \ldots, m+1\}$.
Thus $W$ contains a closed directed walk, which contradicts the fact that $D$ is acyclic.
Hence
$|N^+(C) \cup N^-(C) | \geq m+3$.
Then, since $|N^+(C)|=|N^+(C)|=m$,
$|N^+(C) \cap N^-(C) | \leq m-3$.
Therefore the statement is true.
\end{proof}

\begin{theorem}\label{thm:cycle_component}
Let $\mathcal{G}_{\ell}$ be the set of graphs having the minimum number of components among $(2,2)$ CCE graphs containing a cycle of length $\ell \geq 3$ and $G_{\ell}$ be a graph in $\mathcal{G}_{\ell}$ of the smallest order.
Then the following are true:
\begin{enumerate}[(i)]
\item $G_{\ell}$ contains at least six isolated vertices;
\item The graph $G_3$ is the only graph  in $\bigcup_{\ell=3}^{\infty}\mathcal{G}_{\ell}$ with a unique nontrivial component that is a cycle and exactly six isolated vertices.
\end{enumerate}
\end{theorem} 
\begin{proof}
Fix an integer $\ell \geq 3$.
For notational convenience, we simply write $G$ for $G_\ell$.
Let $D$ be a minimal digraph of $G$.
Take a sink $x$ in $D$.
Then $x$ is isolated in $CCE(D)$ and so, by Proposition~\ref{prop:minimal_property}(i), $x$ cannot have indegree $1$.
If $x$ has indegree $0$, then $CCE(D-x)$ is a graph having less components than $G$ and $CCE(D-x)$ still has a cycle of length $\ell$, which contradicts the choice of $G$.
Thus $x$ has indegree $2$.
Hence the predators of $x$ are adjacent in $G$ by Proposition~\ref{prop:minimal_property}(ii).
Suppose that the predators of $x$ lie on a path component $P$ in $G$.
Then the predators have no common prey in $D-x$ by Proposition~\ref{prop:minimal_property}(iii) and so they are not adjacent in $CCE(D-x)$.
Therefore the component $P$ breaks up into two pieces in $CCE(D-x)$ while one component disappears by deleting $x$.
Thus $CCE(D-x)$ has the same number of components as $CCE(D)$.
By the way, $CCE(D-x)$ still has a cycle of length $\ell$, which contradicts the choice of $G$.
Therefore the predators of $x$ lie on a cycle component in $G$.
 Since $x$ was arbitrarily chosen,
 we conclude that
\statement{sta:thm:cycle_component_1}{each sink in $D$ has two predators which are consecutive vertices on a cycle in $G$.}
Thus each predator of a sink has outdegree $2$ by Proposition~\ref{prop:property_degree2}(i) and so 
\statement{sta:thm:cycle_component_2}{each predator of a sink has a prey distinct from the sink in $D$.} 
If $D$ has exactly one sink $x$,
then $D-x$ has no sink by \eqref{sta:thm:cycle_component_2} and so $D-x$ is not acyclic, which is impossible.
Thus $D$ has at least two sinks.
By the way, we may show that $D$ has at least three sinks. 
To show it by contradiction, suppose that $x$ and $x'$ are the only sinks in $D$.
 If there is no common predator of $x$ and $x'$, then $D-\{x,x'\}$ has no sink by \eqref{sta:thm:cycle_component_2}, which is impossible.
 Thus there exists a common predator $y$ of $x$ and $x'$.
 Then $y$ lies on a cycle in $G$ by \eqref{sta:thm:cycle_component_1}.
Thus $N^-_D(x)=\{y,y'\}$ and
 $N^-_D(x')=\{y,y''\}$ where $y'yy''$ is a section of the cycle.
Note that $y'$ and $y''$ are distinct by Proposition~\ref{prop:minimal_property}(iii). 
Since $yy'$ and $yy''$ are edges of $G$,
$y$ has two predators $z_1$ and $z_2$ such that \[N^+_D(z_1)=\{y,y'\}\quad  \text{and} \quad N^+_D(z_2)=\{y,y''\}.\]
By the assumption that $x$ and $x'$ are the only sinks in $D$, the sinks of $D':=D-\{x,x',y\}$
 belong to $N^-_D(x)\cup N^-_D(x')\cup N^-_D(y)-\{x,x',y\}=\{y',y'',z_1,z_2\}$.
 However, none of these can be a sink of $D'$.
 For, it is easy to check that $z_1$ and $z_2$ are not sinks of $D'$.
 Since each of $y'$ and $y''$ has degree $2$ in $G$,
 each of $y'$ and $y''$ has outdegree $2$ by Proposition~\ref{prop:property_degree2}(i) and so has a prey not belonging to $\{x,x'\}$.
 Moreover, $y' \not \to y$ and $y'' \not \to y$ by Proposition~\ref{prop:(2,2)}.
 Therefore $y'$ and $y''$ are not sinks in $D'$ and so $D'$ has no sinks, which is impossible.
 Thus $D$ has at least three sinks.
 Hence $D$ has at least three sources by Proposition~\ref{prop:reverse_property_minimal}. 
By \eqref{sta:thm:cycle_component_1}, each sink of $D$ is not a source.
Therefore $D$ has at least three sinks and at least three sources.
Thus $G$ has at least $6$ isolated vertices and so part (i) is true.
Let $C$ be a cycle of length $\ell$ in $G$.
Then
\statement{sta:thm:cycle_component_3}{$C \cup 6P_1$ is an induced subgraph of $G$.} 
Since
$V(C) \cap (N^+(C)\cup N^-(C))=\emptyset$ by Proposition~\ref{prop:(2,2)}, $|V(G)|\geq |V(C)| + |N^+(C)\cup N^-(C)|$.
Then, since $|N^+(C)\cup N^-(C)| \geq \ell+3$ by Lemma~\ref{lem:cycle-char},
\[|V(G)|\geq |V(C)| + |N^+(C)\cup N^-(C)| \geq 2\ell+3.\]
 Therefore $G_3\cong C_3 \cup 6P_1$ by \eqref{sta:thm:cycle_component_3} since $CCE(D_5)\cong C_3 \cup 6P_1$ (see the digraph $D_5$ and its CCE graph $CCE(D_5)$ given in Figure~\ref{fig:interval}).
Suppose $\ell\geq 4$.
Then
\[|V(G)| \geq 2\ell+3 > \ell+6=|V(C)|+|V(6P_1)|\] and so $G$ must contain a component not belonging to $C\cup 6P_1$.
Thus part (ii) is true.
\end{proof}

			\begin{figure}
			\begin{center}
			\subfigure[$D_5$]{
\resizebox{0.3\textwidth}{!}{%
\begin{tikzpicture}[scale=1] \tikzset{mynode/.style={inner sep=2pt,fill,outer sep=2.3pt,circle}}
            \node [mynode] (u12) at (-1.5,1) [label=left :$u_1$] {};
			\node [mynode] (u23) at (0,1) [label=above :$u_2$] {};
			\node [mynode] (u31) at (1.5,1) [label=right :$u_3$] {};
		    \node [mynode] (u1) at (-1.5,0) [label=left :$v_1$] {};
		    \node [mynode] (u2) at (0,0) [label=right :$v_2$] {};
		    \node [mynode] (u3) at (1.5,0) [label=right :$v_3$] {};
		    \node [mynode] (u+12) at (-1.5,-1) [label=left :$x_1$] {};
		    \node [mynode] (u+23) at (0,-1) [label=below:$x_2$] {};
			\node [mynode] (u+31) at (1.5,-1) [label=right :$x_3$] {};		

\draw[->, thick] (u1) edge[] (u12);
\draw[->, thick] (u1) edge[] (u31);
\draw[->, thick] (u2) edge[] (u12);
\draw[->, thick] (u2) edge[] (u23);
\draw[->, thick] (u3) edge[] (u23);
\draw[->, thick] (u3) edge[] (u31);

\draw[->, thick] (u+12) edge[] (u1);
\draw[->, thick] (u+12) edge[] (u2);
\draw[->, thick] (u+23) edge[] (u2);
\draw[->, thick] (u+23) edge[] (u3);
\draw[->, thick] (u+31) edge[] (u1);
\draw[->, thick] (u+31) edge[] (u3);

			\end{tikzpicture}
}
}
\hspace{1.47cm}
\subfigure[$CCE(D_5)$]{
\resizebox{0.3\textwidth}{!}{%
\begin{tikzpicture}[scale=1] \tikzset{mynode/.style={inner sep=2pt,fill,outer sep=2.3pt,circle}}
			\node [mynode] (u12) at (-1.5,1) [label=left :$u_1$] {};
			\node [mynode] (u23) at (0,1) [label=above :$u_2$] {};
			\node [mynode] (u31) at (1.5,1) [label=right :$u_3$] {};
		    \node [mynode] (u1) at (-1.5,0) [label=left :$v_1$] {};
		    \node [mynode] (u2) at (0,0) [label=below :$v_2$] {};
		    \node [mynode] (u3) at (1.5,0) [label=right :$v_3$] {};
		    \node [mynode] (u+12) at (-1.5,-1) [label=left :$x_1$] {};
		    \node [mynode] (u+23) at (0,-1) [label=below:$x_2$] {};
			\node [mynode] (u+31) at (1.5,-1) [label=right :$x_3$] {};	

\draw[thick] (u1) edge[] (u2);
\draw[thick] (u2) edge[] (u3);
\draw[thick] (u3) to [in=45, out=135, distance=0.5cm] (u1);
			\end{tikzpicture}
}
}
			\end{center}
			\caption{Digraphs and its CCE graphs in the proof of Theorem~\ref{thm:cycle_component}}
			\label{fig:interval}
			\end{figure}
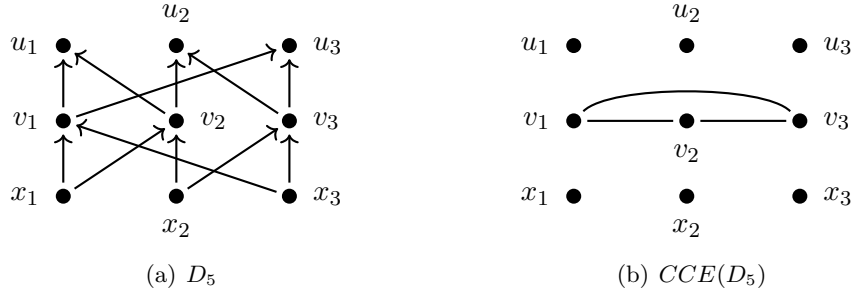

\begin{proof}[A proof of Theorem~\ref{thm:interval}]
Suppose that
$G$ is not interval.
Then, by Proposition~\ref{prop:degree},
$G$ contains a cycle component of length $\ell \geq 4$.
Let $\mathcal{G}_{\ell}$ be the set of graphs having the least components among the $(2,2)$ CCE graphs containing a cycle of length $\ell$ and $G_{\ell}$ be a graph in $\mathcal{G}_{\ell}$ with the least order.
Then, by (i) and (ii) of Theorem~\ref{thm:cycle_component},
$G_{\ell}$ contains at least eight components.
Thus $t \geq 8$. 
Therefore we have shown that any $(2,2)$ CCE graph with at most seven components is interval.
We note that
for the digraph $D_6$ given in Figure~\ref{fig:interval},
$CCE(D_6) \cong C_4 \cup 7P_1$, which has eight components.
Since $CCE(D_6)$ is not an interval graph, the inequality is tight.
\end{proof}

	\begin{figure}
			\begin{center}
\subfigure[$D_6$]{
\resizebox{0.4\textwidth}{!}{%
\begin{tikzpicture}[scale=0.8] 
\tikzset{mynode/.style={inner sep=2pt,fill,outer sep=2.3pt,circle}}

			\node [mynode] (u12) at (-1.5,1) [label=left :$u_1$] {};
			\node [mynode] (u23) at (0,1) [label=above :$u_2$] {};
			\node [mynode] (u31) at (1.5,1) [label=right :$u_3$] {};
		    \node [mynode] (u1) at (-3,0) [label=left :$v_1$] {};
		    \node [mynode] (u2) at (-1.5,0) [label=left :$v_2$] {};
		    \node [mynode] (u3) at (1.5,0) [label=right :$v_3$] {};
		    \node [mynode] (u4) at (3,0) [label=right :$v_4$] {};
		    \node [mynode] (w) at (0,0) [label=below :$w$] {};
		    \node [mynode] (u+12) at (-1.5,-1) [label=left :$x_1$] {};
		    \node [mynode] (u+23) at (0,-1) [label=below:$x_2$] {};
			\node [mynode] (u+31) at (1.5,-1) [label=right :$x_3$] {};

\draw[->,thick] (u1) to [in=135, out=90, distance=0.5cm](w);
 \draw[->,thick] (u4) to [in=45, out=90, distance=0.5cm](w);
	  \draw[->,thick] (w) edge[](u2);
	    \draw[->,thick] (w) edge[](u3);

\draw[->, thick] (u1) edge[] (u12);
\draw[->, thick] (u2) edge[] (u12);
\draw[->, thick] (u2) edge[] (u23);
\draw[->, thick] (u3) edge[] (u23);
\draw[->, thick] (u3) edge[] (u31);
\draw[->, thick] (u4) edge[] (u31);
\draw[->, thick] (u+12) to [in=305, out=155, distance=0.5cm] (u1);
\draw[->, thick] (u+12) edge[] (u2);

\draw[->, thick] (u+23) edge[] (u1);
\draw[->, thick] (u+23) edge[] (u4);

\draw[->, thick] (u+31)  to [in=235, out=25, distance=0.5cm] (u4);
\draw[->, thick] (u+31) edge[] (u3);

			\end{tikzpicture}
}
}
\subfigure[$CCE(D_6)$]{
\resizebox{0.4\textwidth}{!}{%
\begin{tikzpicture}[scale=0.8] 
\tikzset{mynode/.style={inner sep=2pt,fill,outer sep=0.5pt,circle}}
			\node [mynode] (u12) at (-1.5,1) [label=left :$u_1$] {};
			\node [mynode] (u23) at (0,1) [label=above :$u_2$] {};
			\node [mynode] (u31) at (1.5,1) [label=right :$u_3$] {};
		    \node [mynode] (u1) at (-3,0) [label=left :$v_1$] {};
		    \node [mynode] (u2) at (-1.5,0) [label=above :$v_2$] {};
		    \node [mynode] (u3) at (1.5,0) [label=above :$v_3$] {};
		    \node [mynode] (u4) at (3,0) [label=right :$v_4$] {};
		    \node [mynode] (w) at (0,0) [label=right :$w$] {};
		    \node [mynode] (u+12) at (-1.5,-1) [label=left :$x_1$] {};
		    \node [mynode] (u+23) at (0,-1) [label=below:$x_2$] {};
			\node [mynode] (u+31) at (1.5,-1) [label=right :$x_3$] {};

\draw[thick] (u1) edge[] (u2);
\draw[thick] (u2) to [in=135, out=45, distance=0.5cm] (u3);
\draw[thick] (u3) edge[] (u4);
\draw[thick] (u4) to [in=315, out=225, distance=0.5cm]  (u1);
			\end{tikzpicture}
}
}
			\end{center}
			\caption{Digraphs and its CCE graphs in the proof of Theorem~\ref{thm:interval}}
			\label{fig:interval}
			\end{figure}
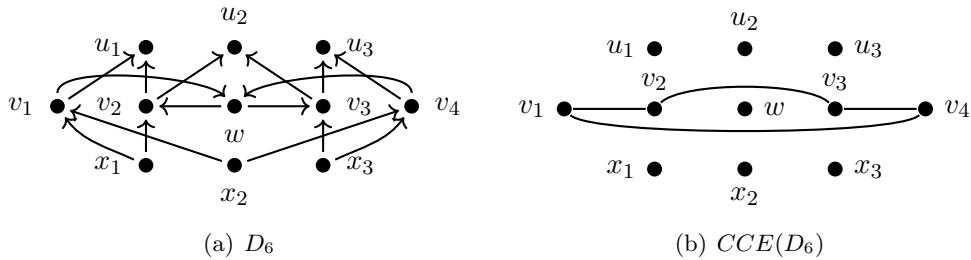

%
 \vspace{3mm}
\noindent {\bf Funding Information} 
This work was supported by Science Research Center Program through the National Research Foundation of Korea(NRF) grant funded by the Korean Government (MSIT)(NRF-2022R1A2C1009648 and NRF-2017R1E1A1A03070489).
 
 \vspace{3mm}
\noindent {\bf Data Availability Statement} Data sharing is not applicable to this article as no datasets were generated or analyzed during the current study.

\section*{Declarations}

\noindent{\bf Conflict of interest}
The authors have no relevant financial or non-financial interests to disclose. The authors declare that they have no conflict of interest.

\end{document}